\documentclass{amsart}

\newtheorem{theorem}{Theorem}[section]
\newtheorem{lemma}[theorem]{Lemma}
\newtheorem{proposition}[theorem]{Proposition}
\newtheorem{corollary}[theorem]{Corollary}

\theoremstyle{definition}
\newtheorem{definition}[theorem]{Definition}
\newtheorem{example}[theorem]{Example}

\newtheorem{problem}[theorem]{Problem}

\theoremstyle{remark}
\newtheorem{remark}[theorem]{Remark}
\newtheorem*{acknowledgements}{Acknowledgements}

\numberwithin{equation}{section}

\newcommand{\ind}{{\textstyle{\mathop{\mathrm{ind}}}}}

\begin{document}

\title{Nonsmoothable group actions on spin $4$-manifolds}

\author{Kazuhiko KIYONO}
\address{Graduate School of Mathematical Sciences,
The University of Tokyo,
3-8-1 Komaba, Meguro-ku, Tokyo 153-8914, Japan.}
\curraddr{}
\email{kiyonok@ms.u-tokyo.ac.jp}
\thanks{}

\subjclass[2000]{Primary 57M60; Secondary 57R57}

\date{}

\dedicatory{}

\begin{abstract}
We show
 that every closed, simply connected, spin topological 4-manifold 
 except $S^4$ and $S^2\times S^2$ 
 admits a homologically trivial, pseudofree, locally linear action of $\mathbb Z_p$ 
 for any sufficiently large prime number $p$ 
 which is nonsmoothable for any possible smooth structure.
\end{abstract}

\maketitle

\section{Introduction}
In this article, 
 we call 
 a locally linear action of a group 
 on a topological manifold 
 \textit{nonsmoothable} 
 if the action is 
 not smooth with respect to any possible smooth structure.
Several authors have been investigated 
 examples of nonsmoothable group actions on 4-manifolds 
 \cite{KLee, KLaw, HL, B, N}.

We restrict our attention 
 to actions of the cyclic groups of odd prime order 
 which are homologically trivial and pseudofree.
A.~L.~Edmonds constructed 
 such actions on all simply connected 4-manifolds \cite{E}.
The main purpose of this article is to show
  that there is a family of 
 locally linear actions constructed 
 by Edmonds's method which are nonsmoothable.

\begin{theorem}\label{thm:main}
Let $X$ be a closed, simply connected, spin topological $4$-manifold 
 not homeomorphic to either $S^4$ or $S^2\times S^2$.
Then, 
 for any sufficiently large prime number $p$, 
 there exists a homologically trivial, pseudofree, 
 locally linear action of $\mathbb Z_p$ on $X$ 
 which is nonsmoothable.
\end{theorem}

The conclusion of Theorem~\ref{thm:main} does not hold for $S^4$.
It is known 
 that every pseudofree locally linear action 
 of odd order cyclic group on $S^4$ 
 is smooth with respect to a smooth structure
 isomorphic to the standard one (see \cite{Wil0}).
Concerning smooth actions on $S^2 \times S^2$, 
 M.~Klemm obtained partial results \cite{Kl},
 while I do not know whether $S^2\times S^2$ admits 
 a homologically trivial, pseudofree, nonsmoothable locally linear action 
 or not.
The following problem seems open.
\begin{problem}
Is there a homologically trivial, pseudofree, nonsmoothable 
 locally linear  action of $\mathbb Z_p$ on $S^2 \times S^2$ 
 for some odd prime number $p$?
\end{problem}

Let $NS(X)$ be the set of every prime number $p$ 
 for which $X$ admits 
 a homologically trivial, pseudofree, 
 nonsmoothable locally linear action of $\mathbb Z_p$.
Theorem~\ref{thm:main} tells that the complement of $NS(X)$ 
 in the set of prime numbers 
 is bounded for each closed, simply connected, spin $4$-manifold $X$ 
 if $X$ is not homeomorphic to $S^4$ or $S^2 \times S^2$.
In the proof of Theorem~\ref{thm:main}, 
 we obtain an estimate of the maximum value of the complement.
\begin{theorem} \label{an estimate}
For any closed, simply connected, spin $4$-manifold $X$ 
 not homeomorphic to $S^4$ or $S^2\times S^2$, 
 $NS(X)$ contains all the prime numbers $p$ satisfying 
 \begin{equation}\label{eq:zatu}
  p\geq12\left[\frac{\max\{b_2^+(X),b_2^-(X)\}+1}2\right]-5.
 \end{equation}
\end{theorem}
\noindent
Here $[x]$ is the maximum integer less than or equal to $x$.
Though
 we need to fix an orientation of $X$ 
 to define $b_2^+(X)$ and $b_2^-(X)$,
 the right-hand side of the above estimate of $p$ 
 does not depend on the choice.

The above estimate is not best possible.
We show a better estimate 
 for the connected sums of the copies of $S^2\times S^2$.
\begin{theorem}\label{thm:sharp}
1. $NS(S^2\times S^2 \# S^2 \times S^2)$ 
 contains all the prime numbers $p \geq 7$.

\noindent
2. For $n\geq 3$, 
 $NS(\#^n S^2 \times S^2)$ 
 contains all the prime numbers $p\geq 19$.
\end{theorem}
We also obtain
\begin{theorem}\label{thm:K3}
$11 \in NS(K3)$.
\end{theorem}

We prove Theorem \ref{thm:main} in three steps.
In section 2
 we give a family of 
 homologically trivial, pseudofree, locally linear actions, 
 slightly modifying the construction of Edmonds in \cite{E} 
 and making use of the criterion of 
 Edmonds and  J.~H.~Ewing in \cite{EE}.
In section 3 
 we calculate the dimension of $\mathbb Z_p$-invariant part 
 of the $\mathbb Z_p$-index of the Dirac operator 
 for the action constructed in Section 2, 
 assuming that $X$ is spin and that the action is smooth 
 for some smooth structure 
 (Proposition~\ref{prop:ind}).
The dimension is equal to the index of the Dirac operator 
 on the quotient $V$-manifold $X/\mathbb Z_p$.
In section 4 
 we prove Theorem \ref{thm:main} 
 applying the $10/8$-type inequality 
 for the quotient $V$-manifold $X/\mathbb Z_p$ in \cite{FF}.
\begin{remark}
Presumably 
 the estimate in Theorem~\ref{thm:sharp} 
 could be improved further in general.
We also do not know the set $NS(K3)$ exactly 
 while Theorem~\ref{an estimate} tells 
 that $NS(K3)$ contains all the prime numbers greater than $113$.
\end{remark}
\begin{remark}
When a smooth structure is endowed 
 on a topological manifold, 
 a locally linear group action on the topological manifold 
 is called nonsmoothable 
 if the action is not smooth 
 with respect to any smooth structure isomorphic to the given one.
W.~Chen and S.~Kwasik 
 constructed group actions on $K3$ surface of this type,
 which are smooth with respect to the standard smooth structure 
 but not smooth with respect to infinitely many exotic structures \cite{CK}. 
X.~Liu and N.~Nakamura constructed 
 group actions on elliptic surfaces 
 which are not smooth with respect to infinitely many smooth structures 
 including the standard smooth structure \cite{LN1, LN2}.
It is not known 
 whether the examples of Liu and Nakamura are nonsmoothable 
 for every smooth structure 
 or not.
Liu and Nakamura 
 used mod-$p$ vanishing theorem of Seibert-Witten invariant 
 for 4-manifolds with non-vanishing Seiberg-Witten invariant.
Recently Nakamura applied a similar method to $K3 \# K3$, 
 for which the Seiberg-Witten invariant is zero 
 but its cohomotopy refinement does not vanish \cite{N}.
We also use Seiberg-Witten theory 
 to investigate nonsmoothability of finite group action.
Our approach is 
 to apply an equivariant version of $10/8$-inequality 
to spin $4$-manifolds, 
which does not depend on non-vanishing of Seiberg-Witten invariant.
\end{remark}

\begin{acknowledgements}
The author thanks 
 Mikio Furuta for his invaluable advice and encouragement,
 Allan 
 Edmonds for information  on 
 pseudofree, locally linear group actions on $S^4$, and
 Yukio Kametani and Nobuhiro Nakamura
 for helpful discussion.
\end{acknowledgements}

\section{Locally linear actions}
Let $X$ be a closed, oriented, simply connected 
 topological $4$-manifold not necessarily spin.
Edmonds proved the following theorem. 
\begin{theorem}
[{\bf Edmonds}, {\sc Theorem} 6.4 in \cite{E}]\label{thm:exist}
For any prime number $p$ not less than $5$, 
 there exists a homologically trivial, pseudofree, locally linear action
 of $\mathbb Z_p$ on $X$.
\end{theorem}
Edmonds constructed the group action 
 using equivariant surgery 
 on the connected sum of $b_2^+(X)$-copies of $\mathbb C P^2$ 
 and 
 $b_2^-(X)$-copies of $\overline{\mathbb C P^2}$ 
 for some choice of $\mathbb Z_p$-action.
Moreover Edmonds and Ewing 
 obtained a necessary and sufficient criterion 
 for realizability of a pair of 
 a fixed point data and a unimodular quadratic form with $\mathbb Z_p$-action
 by a pseudofree locally linear $\mathbb Z_p$-action on $X$ 
 \cite{EE}.
In this section 
 we follow Edmonds's construction 
 with a slight modification 
 to obtain a family of fixed point data 
 satisfying Edmonds and Ewing's criterion.
More specifically,
 we make realizable fixed point data 
 by gathering the fixed point data of 
 pseudofree $\mathbb Z_p$-actions 
 on $\mathbb C P^2$, $\overline{\mathbb C P^2}$ and $S^4$.

We identify $\mathbb Z_p$ with the subgroup 
 of $U(1)$,
 and,
 for an integer $a$,
 let $\mathbb C_a$ be the one-dimensional complex representation of $\mathbb Z_p$ 
 defined by $z\mapsto g^az$ 
 for $z\in\mathbb C$ and $g\in\mathbb Z_p$.

\begin{definition}\label{defn:linear}
Using weights $\alpha=(a_0,a_1,a_2)$, $\alpha'=(a'_0,a'_1,a'_2)$ 
 and $\beta=(b_1,b_2)$ respectively,
 define $\mathbb Z_p$-manifolds 
 $\mathbb C P^2_\alpha$, $\overline{\mathbb C P^2_{\alpha'}}$ and $S^4_\beta$ 
 as $\mathbb C P^2$, $\overline{\mathbb C P^2}$ and $S^4$ with pseudofree $\mathbb Z_p$-actions 
 as follows.
\begin{enumerate}
\item 
Suppose $a_0$, $a_1$ and $a_2$ are integers 
 which are not congruent modulo $p$ each other.
Let $\mathbb C P^2_\alpha$ denote the quotient space 
 $(\mathbb C_{a_0}\oplus\mathbb C_{a_1}\oplus\mathbb C_{a_2}\setminus\{(0,0,0)\})/\mathbb C^\ast$.
\item
Suppose $a'_0$, $a'_1$ and $a'_2$ are integers 
 which are not congruent modulo $p$ each other.
Let $\overline{\mathbb C P^2_{\alpha'}}$ denote 
 the same as $\mathbb C P^2_{\alpha'}$ 
 but with the opposite orientation.
\item
Suppose $b_1$ and $b_2$ are integers not congruent to 0 modulo $p$ either.
Let $S^4_\beta$ denote the unit sphere of $\mathbb C_{b_1}\oplus\mathbb C_{b_2}\oplus\mathbb R$,
 where $\mathbb R$ is the trivial one-dimensional real representation of $\mathbb Z_p$.
\end{enumerate}
\end{definition}
Note that the two weights $\alpha_0=(a_0, a_1, a_2)$ and
 $\alpha_1=(a_0+1,a_1+1,a_2+1)$ give 
 the same action on $\mathbb C P^2$, 
 hence $\mathbb C P^2_{\alpha_0}$ and $\mathbb C P^2_{\alpha_1}$ are 
 $\mathbb Z_p$-equivariantly diffeomorphic.
From now on 
 we assume that $a_0+a_1+a_2$ and $a'_0+a'_1+a'_2$ are even 
 for weights $\alpha=(a_0,a_1,a_2)$ and $\alpha'=(a'_0,a'_1,a'_2)$ 
 in Definition~\ref{defn:linear} 
 without loss of generality.

To make a realizable fixed point data 
 by gathering those of 
 $\mathbb C P^2_\alpha$'s, $\overline{\mathbb C P^2_{\alpha'}}$'s and $S^4_\beta$'s 
 we may need to reduce the number of fixed points.

\begin{definition}\label{defn:cancel}
We call a pair of fixed points \textit{a cancelling pair} 
 if there is a weight $\beta$ 
 such that the fixed point data of $S^4_\beta$ 
 coincides with that of the pair.
We also call 
 such a weight $\beta$ \textit{a weight of the cancelling pair}.
\end{definition}

A pair of fixed points is a cancelling pair 
 if and only if 
 the two isotropy representations at the two fixed points
 are isomorphic 
 to each other through 
 an orientation-reversing isomorphism.
The weight of the cancelling pair is 
 one of the weights of these representations.
(We have two possible representatives
 of weights for each cancelling pair.)

We will use the following cancelling pairs later.
 These examples are special cases of 
 {\sc Lemma} 6.2 in \cite{E}.
\begin{example}\label{ex:cancel}
Let $p$ be a prime number not less than 5.
\begin{enumerate}
\item Let $a,b$ and $c$ be integers 
      satisfying $a-b,a-b-c,a-b-2c,c \not\equiv 0 \bmod p$
      and $a \equiv c \bmod 2$.
      For $\alpha_1=(a,b,b+c)$ and $\alpha_2=(a,b+c,b+2c)$,
      the pair 
      $[0,0,1]$ on $\mathbb C P^2_{\alpha_1}$ 
      and
      $[0,1,0]$ on $\mathbb C P^2_{\alpha_2}$ 
      is a cancelling pair.
\item Let $i$ be an integer 
      satisfying $i \not\equiv -1,-2,-3 \bmod p$.
      For $\alpha_1=(-1,i,i+1)$ and $\alpha_2=(-1,i+1,i+2)$, 
      the pair of 
      $[0,0,1]$ on $\mathbb C P^2_{\alpha_1}$ 
      and
      $[0,1,0]$ on $\mathbb C P^2_{\alpha_2}$
      is a cancelling pair.
\item For $\alpha_1=(-1,p-4,p-3)$ and $\alpha_2=(-1,0,1)$,
      the pair of 
      $[0,0,1]$ on $\mathbb C P^2_{\alpha_1}$ 
      and
      $[1,0,0]$ on $\mathbb C P^2_{\alpha_2}$
      is a cancelling pair.
\item Let $n$ be a positive integer.
      For each $1\leq i \leq n$
      we write $R(i)$ for the remainder 
      of $i$ divided by $p-3$, and
      $\alpha_i$ for the weight $(-1,R(i),R(i)+1)$.
      There are $n-1$ cancelling pairs 
      in  the fixed points  
      of the disjoint union
      $\coprod_{1\leq i\leq n}\mathbb C P^2_{\alpha_i}$.
      \end{enumerate}
Note that
 (2) is the special case of (1) with $a=-1$, $b=i$, and $c=1$,
(3) is essentially the case (2) with $i=p-4$
      since weights $(-1,0,1)$ and $(p-3,p-2,p-1)$ induce
      the same action,
and (4) is a consequence of the cases (2) and (3).
\end{example}

\begin{proposition}\label{prop:exist}
Let $X$ be a closed, oriented, simply connected topological $4$-manifold,
 and $m$, $m'$, $r$ and $s$ non-negative integers satisfying
 \[
  m-m'=\sigma(X)\quad\textrm{and}\quad
  3(m+m')+2(r+s)=\chi(X),
 \]
 where $\sigma(X)$ and $\chi(X)$ 
 are the signature and the Euler number of $X$ 
respectively.
Suppose there are weights
 ${\alpha_i}$  $(1\leq i\leq m)$,
 ${\alpha'_j}$ $(1\leq j\leq m')$, 
 and ${\beta_k}$ $(1\leq k\leq r)$ such
 that the fixed point set of the disjoint union
 \[
  \left(\coprod_{1\leq i\leq m} \mathbb C P^2_{\alpha_i}\right)
  \coprod
  \left(\coprod_{1 \leq j \leq m'} 
  \overline{\mathbb C P^2_{\alpha'_j}}\right)
  \coprod
  \left(\coprod_{1 \leq k \leq r}S^4_{\beta_k}\right)
 \]
 has 
 $s$ cancelling pairs. 
Let $\mathcal D$ be the fixed point data 
 for those fixed points 
 which does not appear in the $s$ cancelling pairs.
Then 
 there exists a homologically trivial, pseudofree, 
 locally linear action of $\mathbb Z_p$ on $X$ 
 whose fixed point data is the same as $\mathcal D$.
\end{proposition}
\begin{proof}
We check 
 that the data $\mathcal D$ satisfies the three conditions 
 REP, GSF, and TOR in \cite{EE}.
We write $Y$ for the disjoint union in the statement 
 of the theorem.
Let $\gamma_l$ $(1\leq l\leq s)$ 
 be weights of the $s$ cancelling pairs on $Y$, 
 and we write $Z$ for the disjoint union 
 \[
  Z=\coprod_{1 \leq l \leq s}S^4_{\gamma_l}.
 \]

Since the number of fixed points of $\mathcal D$ is 
 $3m+3m'+2r-2s=\chi(X)$, 
 $\mathcal D$ satisfies the condition REP 
 for homologically trivial action on $X$.

The right-hand side of GSF for $\mathcal D$ is 
 the difference between those for the fixed point data of $Y$ and of $Z$.
This is equal to $\sigma(Y)-\sigma(Z)$
 since the condition GSF is true for both $Y$ and $Z$.
By the assumption of the proposition, 
 $\sigma(Y)-\sigma(Z)=m-m'=\sigma(X)$.
So $\mathcal D$ satisfies also the condition GSF 
 for homologically trivial action on $X$.

The condition TOR is equivalent to 
 the equation of Application~8.6 in \cite{EE} 
 for homologically trivial action.
Let $L(\mathcal D)$, $L(Y)$, and $L(Z)$ be 
 the left-hand sides of the equations for $\mathcal D$, 
 for the fixed point data of $Y$, and for that of $Z$ respectively.
Note that we have $L(\mathcal D) =L(Y)/L(Z)$.
Since the fixed point data of $Y$ and that of $Z$ satisfy
\[
 L(Y)\approx(-1)^m\times(-(-1))^{m'}\times(-1)^r=(-1)^{m+r}
 \quad\textrm{and}\quad
 L(Z)\approx(-1)^s,
\]
we obtain
\[
 L(\mathcal D)\approx(-1)^{m+r-s}=(-1)^{b_2^-(X)+1},
\]
which is the equation for $\mathcal D$.
\end{proof}

Edmonds used the construction with $(m,m',r,s)=(b_2^+(X),b_2^-(X),0,b_2(X)-1)$ 
 to prove Theorem~\ref{thm:exist} in \cite{E}.
We will use other choices of $(m,m',r,s)$
 to prove Theorems~\ref{thm:main}, \ref{thm:sharp} and \ref{thm:K3}.

\section{Index of Dirac operator}
In this section, 
 we calculate the dimension 
 of $\mathbb Z_p$-invariant part 
 of the $\mathbb Z_p$-index 
 of the Dirac operator 
 on $X$ 
 for the $\mathbb Z_p$-action given in
Proposition~\ref{prop:exist},
 assuming that 
 $X$ is a spin smooth manifold 
 and that the $\mathbb Z_p$-action is smooth.

Recall that 
 we are assuming that $a_0+a_1+a_2$ is even 
 for a weight $\alpha=(a_0, a_1, a_2)$.
Let $|\alpha|$ be $a_0+a_1+a_2$.
\begin{definition}\label{defn:N}
Define a non-negative integer $N(p,\alpha)$ as 
 the number of ordered triplets of integer $(n_0, n_1, n_2)$ 
 satisfying 
 \[
  n_0, n_1, n_2\geq0,\quad
  n_0+n_1+n_2=\frac{p-3}2, \;\;\textrm{and}\;\;
  a_0n_0+a_1n_1+a_2n_2+\frac{|\alpha|}2\equiv0 \mod p.
 \]
\end{definition}


\begin{lemma}\label{lem:dolbeault}
The dimension of $\mathbb Z_p$-invariant part of the $\mathbb Z_p$-index of 
 the Dolbeault operator on $\mathbb C P^2_\alpha$ 
 with coefficient $\mathcal O(\frac{p-3}2)\otimes\mathbb C_{-\frac{|\alpha|}2}$
 is equal to $N(p,\alpha)$.
\end{lemma}
\begin{proof}
Let $L_\alpha$ be the $\mathbb Z_p$-equivariant line bundle 
 $\mathcal O(\frac{p-3}2)\otimes\mathbb C_{-\frac{|\alpha|}2}$.
The $\mathbb Z_p$-index of the twisted Dolbeault operator is equal to
 \[
  H^0(\mathbb C P^2;\mathcal O(L_\alpha))
  =\textrm{Span}\left\{z_0^{n_0}z_1^{n_1}z_2^{n_2}\;\left|\;
           n_0,n_1,n_2\geq0,\;n_0+n_1+n_2=\frac{p-3}2\right.\right\}
 \]
 with $\mathbb Z_p$-action
 \begin{align*}
  \mbox{}&
  g[z_0,z_1,z_2;z_0^{n_0}z_1^{n_1}z_2^{n_2}]
  =
  [g^{a_0}z_0,g^{a_1}z_1,g^{a_2}z_2;
   g^{-\frac{|\alpha|}2}z_0^{n_0}z_1^{n_1}z_2^{n_2}]
  \\&=
  [g^{a_0}z_0,g^{a_1}z_1,g^{a_2}z_2;
  g^{-\sum a_in_i-\frac{|\alpha|}2}
  (g^{a_0}z_0)^{n_0}(g^{a_1}z_1)^{n_1}(g^{n_2}z_2)^{n_2}].
 \end{align*}
This completes the proof.
\end{proof}

\begin{proposition}\label{prop:ind}
Let $X$ be a closed, simply connected, spin smooth $4$-manifold.
Suppose a $\mathbb Z_p$-action constructed 
 in Proposition~\ref{prop:exist} 
 is smooth with respect to some smooth structure of $X$.
Then the action has the unique lift to the spin structure on $X$ and
 the dimension of $\mathbb Z_p$-invariant part of the $\mathbb Z_p$-index of 
 the Dirac operator on $X$ is equal to
 \[
  \sum_{i=1}^mN(p,\alpha_i)-\sum_{j=1}^{m'}N(p,\alpha'_j)-\frac{\sigma(X)}8p.
 \]
\end{proposition}
\begin{proof}
Let $Y$ and $Z$ be as in the Proof of Theorem~\ref{prop:exist},
 and $X'$ the disjoint union of $X$ and $Z$.
Note that $Z$ and $X'$ are spin.
In general, since $p$ is odd, 
 $\mathbb Z_p$-actions on spin manifolds
 have unique lift to spin structures. 
%
Let $D_X$, $D_Z$ and $D_{X'}$ be the Dirac operators
 on $X$, $Z$ and $X'$ respectively.
Since $\ind_{\mathbb Z_p}D_Z=0$ 
 we have
 $\ind_{\mathbb Z_p}D_{X'}=\ind_{\mathbb Z_p}D_X$.
We construct below 
 a $\mathbb Z_p$-equivariant spin$^c$-structure on $Y$ 
 so that it is $\mathbb Z_p$-equivariantly spin 
 on a neighborhood of the fixed point set,
 and compare $\ind_{\mathbb Z_p}D_{X'}$ with
 the $\mathbb Z_p$-index of the spin$^c$-Dirac operator of it.

The $\mathbb Z_p$-equivariant spin$^c$-structure 
 on each $\mathbb C P^2_{\alpha_i}$-component of $Y$ is 
 defined so that its spin$^c$-Dirac operator is
 identified with the Dolbeault operator twisted by $L_{\alpha_i}$ 
 in Lemma~\ref{lem:dolbeault}.
If we write $K$ for the canonical line bundle of $\mathbb C P^2$,
 then the square of $L_{\alpha_i}$ is
 $\mathbb Z_p$-equivariantly isomorphic to
 $K \otimes \mathcal O(p)$.
Since $\mathcal O(p)$ is $\mathbb Z_p$-equivariantly trivial
 on a neighborhood of the fixed point set,
 $L_\alpha$ is a $\mathbb Z_p$-equivariant square root of $K$ there.
It implies that 
 the spin$^c$-structure is 
 $\mathbb Z_p$-equivariantly spin on a neighborhood of the fixed point set.
The $\mathbb Z_p$-equivariant spin$^c$-structure 
 on each $\overline{\mathbb C P^2_{\alpha'_j}}$-component of $Y$ is
 defined so that the spin$^c$-Dirac operator is
 identified with the same twisted Dolbeault operator with 
 opposite parity of degree.
The $\mathbb Z_p$-equivariant spin$^c$-structure 
 on each $S^4_{\beta_k}$-component of $Y$ is
 defined as $\mathbb Z_p$-equivariant spin structure.

Let $D_Y$ be the spin$^c$-Dirac operator 
 on the spin$^c$-structure defined as above.
Since the spin action on $X'$ is isomorphic to the spin$^c$ action on $Y$ 
 on neighborhoods of their fixed point sets,
 $\ind_gD_{X'}$ and $\ind_gD_Y$ 
 coincide for $g \neq 1 \in \mathbb Z_p$.
This is a consequence of the
 localization of the equivariant indices
 as elements of some localization of 
 equivariant $K$-groups of the neighborhoods of
 the fixed point sets, or one could also see it
 from the Atiyah-Segal-Lefschetz formula.
Hence we first obtain
 \begin{align} \label{equation 1}
  \dim(\ind_{\mathbb Z_p}D_{X'})^{\mathbb Z_p}-\dim(\ind_{\mathbb Z_p}D_Y)^{\mathbb Z_p} 
  &
  =
  \frac1p\sum_{g\in\mathbb Z_p}\ind_gD_{X'}-
  \frac1p\sum_{g\in\mathbb Z_p}\ind_gD_Y\nonumber
  \\ 
  & =
  \frac1p\left(\ind_1 D_{X'}-\ind_1 D_Y\right)\nonumber\\ 
  &=\frac1p\left(\ind D_{X'}-\ind D_Y\right).
 \end{align}
Secondly, from the Hirzebruch signature theorem and 
 a direct calculation, 
 we have
 \[
  \left\{
  \begin{array}{lll}
   \ind D_{X'}&=&-\dfrac{\sigma(X')}{8}=-\dfrac{\sigma(X)}{8},\\
   \ind D_Y
   &=&
   (m-m')\dim H^0\left(\mathbb C P^2;\mathcal O(\tfrac{p-3}2)\right)
   =
   \sigma(X)\dfrac{p^2-1}8,
 \end{array}\right.
 \]
 which implies
 \begin{equation}\label{equation 2}
  \frac1p\left(\ind D_{X'}-\ind D_Y\right)=
  -\sigma(X)\frac{p}8,
 \end{equation}
Thirdly, applying Lemma~\ref{lem:dolbeault}
 to each component of $Y$, 
 we have
 \begin{equation}  \label{equation 3}
  \dim(\ind_{\mathbb Z_p}D_Y)^{\mathbb Z_p}
  =
  \sum_{i=1}^mN(p,\alpha_i)-\sum_{j=1}^{m'}N(p,\alpha'_j).
 \end{equation}
Now the equations\eqref{equation 1}, \eqref{equation 2} 
 and \eqref{equation 3} imply the required formula.
\end{proof}
\begin{remark}
The dimension of $\mathbb Z_p$-invariant part of the $\mathbb Z_p$-index of 
 the Dirac operator on $X$ is 
 nothing but the index of the Dirac operator 
 on the quotient spin V-manifold $X/\mathbb Z_p$.
\end{remark}

\section{Nonsmoothability}
In this section,
 we prove Theorem~\ref{thm:main} 
 choosing appropriate weights 
 and using 
 {\it the $10/8$-type inequality 
 for the quotient V-manifold $X/\mathbb Z_p$} 
 in \cite{FF}.

\begin{theorem}\label{thm:smooth}
Let $X$ be a closed, simply connected, spin smooth $4$-manifold
 not homeomorphic to $S^4$.
Suppose the integers $m$, $m'$, $r$, $s$ and $\mathbb Z_p$-manifolds
 $\mathbb C P^2_{\alpha_i}$ $(1\leq i \leq m)$,  
 $\overline{\mathbb C P^2_{\alpha'_j}}$ $(1\leq j \leq m')$,
 $S^4_{\beta_k}$ $(1\leq k\leq r)$ satisfy
 the assumption of Proposition~\ref{prop:exist}.
If the $\mathbb Z_p$-action on $X$ constructed in Proposition~\ref{prop:exist}
 is smooth with respect to some smooth structure of $X$, 
 then the inequality
 \[
  -b_2^-(X)<
   \sum_{i=1}^mN(p,\alpha_i)
   -\sum_{j=1}^{m'}N(p,\alpha'_j)
   -\frac{\sigma(X)}8p
  <b_2^+(X)
 \]
 holds.
\end{theorem}
\begin{proof}
In general when a finite group $G$ acts 
 on a closed, spin smooth $4$-manifold $W$ 
 preserving its orientation and the spin structure, 
 Y.~Fukumoto and M.~Furuta \cite{FF} showed the inequality
 \[
  \dim(\ind_GD)^G<\dim_\mathbb R H^2_+(W;\mathbb R)^G
 \]
 when the right-hand side is not zero,
 where $D$ is the $G$-equivariant Dirac operator on $W$.
In our case,
 since the action is homologically trivial,
 the right-hand side for $W=X$ is equal to $b_2^+(X)$.
When $X$ is a spin smooth manifold not homeomorphic to $S^4$,
then a theorem of S.~K.~Donaldson \cite{D0} implies $b_2^+(X),b_2^-(X)>0$.
Therefore 
 if the action is smoothable, we have the above inequality.
Using the formula given by Proposition~\ref{prop:ind},
 we can write the inequality as 
 \[
  \sum_{i=1}^mN(p,\alpha_i)-\sum_{j=1}^{m'}N(p,\alpha'_j)-\frac{\sigma(X)}8p
  <b_2^+(X).
 \]
Reversing the orientation of $X$, 
 we similarly 
 obtain 
 \[
  \sum_{j=1}^{m'}N(p,\alpha'_j)-\sum_{i=1}^mN(p,\alpha_i)+\frac{\sigma(X)}8p
  <b_2^-(X).
 \]
\end{proof}

\begin{proof}[Proof of Theorems~\ref{thm:main} and \ref{an estimate}]
When $X$ has no smooth structure
 Theorem~\ref{thm:main} is included in Edmonds's Theorem~\ref{thm:exist}.
So we assume below 
 that $X$ has a smooth structure.
We also $\sigma(X)\leq0$ 
 giving the opposite orientation to $X$ if necessary.

To construct an action 
 which does not satisfy the inequality of Theorem~\ref{thm:smooth}
 we choose different triplet $(m, m', r, s)$ from 
 that used by Edmonds in \cite{E}.
\begin{lemma}\label{lem:spin}
Let $p$ be a prime number not less than $5$,
 and $X$ a closed, simply connected, spin smooth $4$-manifold 
 with $\sigma(X)\leq0$ not homeomorphic to $S^4$ or $S^2\times S^2$.
Then there exist  weights
 \[
  \alpha'_1,\;\alpha'_2,\; \ldots,\;\alpha'_{-\sigma(X)}
 \]
 for $\overline{\mathbb C P^2}$
 satisfying the following property:
For any weights $\alpha_0$ and $\alpha'_0$,
 the $\mathbb Z_p$-manifolds 
 $\mathbb C P^2_{\alpha_0}$,
 $\overline{\mathbb C P^2}_{\alpha'_j}$ $(0\leq j\leq -\sigma(X))$
 and $S^4_{\beta_k}$ $(1\leq k\leq r)$ 
 for some non-negative integer $r$ 
 and some weights $\beta_k$ $(1\leq k\leq r)$ 
 satisfy
 the assumption of Proposition~\ref{prop:exist}. 
\end{lemma}
\begin{proof}
For each $1\leq j \leq -\sigma(X)$
 we write $R(j)$ for the remainder of $j$ divided by $p-3$. 
We show that the weights $\alpha'_j=(-1,R(j),R(j)+1)$ 
 $(1\leq j \leq -\sigma(X))$ satisfy
 the required property.
Recall that,
 in the assumption of Proposition 12, 
 $m$, $m'$, $r$ and $s$ are non-negative integers
 satisfying $m-m'=\sigma(X)$ and $3(m+m')+2(r+s)=\chi(X)$.

Case I:  If $-3\sigma(X)+6\leq\chi(X)$, then we take
 \[
  m=1,\quad m'=-\sigma(X)+1,\quad 
  r=\frac{3\sigma(X)-6+\chi(X)}2,\quad\textrm{and}\quad s=0.
 \]
Since we do not require existence of cancelling pairs,
 any choice of $\mathbb Z_p$-manifolds satisfies 
 the assumption of Proposition~\ref{prop:exist}.

Case II:
If $-3\sigma(X)+6>\chi(X)$, then we take
 \[
  m=1,\quad m'=-\sigma(X)+1,\quad r=0, \quad\textrm{and}\quad
  s=\frac{-3\sigma(X)+6-\chi(X)}2.
 \]
Our assumption implies $s \geq 0$.
We will show the inequality
 $-\sigma(X)-1 \geq s (\geq 0).$
Then the proof will be completed because
 Example~\ref{ex:cancel} (4) tells that, under the
 inequality $-\sigma(X)-1\geq 0$, the number of
 the cancelling pairs is
 at least $-\sigma(X)-1$, and hence at least $s$
 under the inequality $s\leq-\sigma(X)-1$.
Since $-\sigma(X)-1-s=b^+_2(X)-3$,
 it suffices to show $b_2^+(X) \geq 3$.
If not, 
 and if $X$ is smooth, 
 Donaldson's Theorems B and C in \cite{D} imply 
 $b_2^+(X)=b_2^-(X)\leq2$, i.e., $\sigma(X)=0$ and $\chi(X)=2$, $4$, or $6$.
The case $\chi(X)=6$ is excluded from the assumption 
 $-3\sigma(X)+6>\chi(X)$ of Case II.
The cases $\chi(X)=2$ and $4$ are
 also excluded from our assumption
 that $X$ is not homeomorphic to $S^4$ or $S^2\times S^2$.
\end{proof}

We continue to prove 
 Theorems~\ref{thm:main} and \ref{an estimate}.

Fix a prime number $p$ not less than 5.
We consider the actions constructed in the proof of Lemma~\ref{lem:spin}.
So we use the notation there.
Choose and fix weights $\alpha'_j$ for $1\leq j\leq-\sigma(X)$ 
as in Lemma~\ref{lem:spin}
 so that the union of fixed points of the $\mathbb Z_p$-manifolds 
 $\overline{\mathbb C P^2_{\alpha'_j}}$ for $1\leq j\leq-\sigma(X)$ 
 has $s$ cancelling pairs.
We can choose arbitrarily the rest of weights 
 $\alpha_0$, $\alpha'_0$, and $\beta_k$ for $1\leq k\leq r$.
We would like to choose these weights so that
 the inequality of Theorem~\ref{thm:smooth} is violated.
Since the fixed point data of $S^4_{\beta}$ does not contribute
 to the $\mathbb Z_p$-index of the Dirac operator on $X$,
 what we can effectively control are 
 the two weights $\alpha_0$ and $\alpha'_0$.

%
If we write $I$ for the integer 
 \[
  I=-\sum_{j=1}^{-\sigma(X)}N(p,\alpha'_j)-\frac{\sigma(X)}8p,
 \]
then Proposition~\ref{prop:ind} implies 
 \[
  \dim(\ind_{\mathbb Z_p}D_X)^{\mathbb Z_p}
  =
  I+N(p,\alpha_0)-N(p,\alpha'_0).
 \]
On the other hand, a direct calculation shows that
\[
\begin{array}{lll} 
 N(p,(-1,0,1)) &=& k\qquad \textrm{for} \quad p=4k\pm1,
\\
 &&\\
N(p,(-1,1,2)) &=& \left\{\begin{array}{llcll}
                    l-1&\quad&\textrm{for}&\;&p=12l-5\\
                    l&&\textrm{for}&&p=12l\pm1\\
                    l+1&&\textrm{for}&&p=12l+5
                   \end{array}\right.,
\end{array} 
\]
which implies 
\begin{equation}\label{eq:N-N}
 N(p,(-1,0,1))-N(p,(-1,1,2))=2l\qquad\textrm{for}\quad
 p=12l+q\;\;(q=\pm1, \pm5).
\end{equation}
In particular,
 if we choose $\alpha_0=(-1,0,1)$ and $\alpha'_0=(-1,1,2)$, 
 then we have
 \[
  \dim(\ind_{\mathbb Z_p}D_X)^{\mathbb Z_p}=I+2l\quad\textrm{for}\quad
  p=12l+q\;\;(q=\pm1,\pm5),
 \]
 and 
 if we choose $\alpha_0=(-1,1,2)$ and $\alpha'_0=(-1,0,1)$ 
 then we have 
 \[
  \dim(\ind_{\mathbb Z_p}D_X)^{\mathbb Z_p}=I-2l\quad\textrm{for}\quad
  p=12l+q\;\;(q=\pm1,\pm5).
 \]
Therefore
 at least one of the absolute values of the above two 
 is greater than or equal to $2l$.
Hence 
 if $p$ is large enough to satisfy
$2l\geq\max\{b_2^+(X),b_2^-(X)\}$,
or the inequality \eqref{eq:zatu}, then
 one of the above actions does not satisfy 
the inequality given in Theorem~\ref{thm:smooth}.
This implies the action is nonsmoothable.
\end{proof}

\begin{remark}
1.
Our construction is not available to find
 a nonsmoothable $\mathbb Z_5$-action, even if it exists.
It is because $N(5,\alpha)=1$ 
 for any weight $\alpha$, and hence
 $\dim(\ind_{\mathbb Z_5}D_X)^{\mathbb Z_5}=3\sigma(X)/5$ 
 for any $\mathbb Z_5$-action on any spin $4$-manifold
 constructed in Proposition~\ref{prop:exist}.
This value $3\sigma(X)/5$ satisfies 
 the inequality of Theorem~\ref{thm:smooth}.

2.
Our construction is not available to find
 a nonsmoothable $\mathbb Z_p$-action on $S^2 \times S^2$, 
 even if it exists.
It is because 
 the fixed point data of any action on $S^2\times S^2$ 
 constructed 
 in Proposition~\ref{prop:exist} 
 is realized by a smooth action 
 (Lemma 5.1 in \cite{Wil}).
\end{remark}

\section{Estimate of $p$}
In the proof of Theorems~\ref{thm:main} and \ref{an estimate}
 in the previous section, we made use of particular choices of
 weights. If we use other choices of weights,
 it is likely that we could construct nonsmoothable
 actions for some other prime numbers as well.
%
Theorems~\ref{thm:sharp} and \ref{thm:K3} are 
examples of this kind.

\begin{proof}[Proof of Theorem~\ref{thm:sharp}]
We take $m$, $m'$, $r$ and $s$ 
 satisfying $n+1=3m+r$, $0\leq r \leq 2$, $m'=m$ and $s=0$.
Any choice of $\mathbb C P^2_{\alpha_i}$ $(1\leq i\leq m)$, 
 $\overline{\mathbb C P^2_{\alpha'_j}}$ $(1\leq j\leq m)$,
 and $S^4_{\beta_k}$ $(1\leq k\leq r)$ 
 satisfies the assumption of Proposition~\ref{prop:exist}.

Take $\alpha_i=(-1,0,1)$ for every $1 \leq i \leq m$, and
 $\alpha'_j=(-1,1,2)$ for every $1 \leq j \leq m$.
If this action is smooth with respect to some smooth structure on $X$
 then 
 \begin{align*}
  \dim(\ind_{\mathbb Z_p}D_X)^{\mathbb Z_p}
  &=
  mN(p,(-1,0,1))-mN(p,(-1,1,2))
  \\&=2lm
  \quad\textrm{for}\quad p=12l+q\;\;(q=\pm1,\pm5)
 \end{align*}
 by Proposition~\ref{prop:ind} and the equation \eqref{eq:N-N}.
Hence,
 for any $n$ and $p$ satisfying
 \[
  2lm\geq n=3m+r-1\quad\textrm{for}\;\;p=12l+q\;\;(q=\pm1,\pm5),
 \]
 the action does not satisfy the inequality of Theorem~\ref{thm:smooth}.
 This implies Theorem~\ref{thm:sharp}.
\end{proof}

\begin{proof}[Proof of Theorem~\ref{thm:K3}]
Let $p$ be $11$ 
 and $X$ the topological manifold homeomorphic to $K3$ surface.

For $\alpha'_1=(-1,1,2)$, $\alpha'_2=(-1,2,3)$, and $\alpha'_3=(-1,3,4)$,
 the three $\mathbb Z_p$-manifolds $\overline{\mathbb C P^2_{\alpha'_1}}$,
 $\overline{\mathbb C P^2_{\alpha'_2}}$, and 
 $\overline{\mathbb C P^2_{\alpha'_3}}$ have two cancelling pairs, 
 as in Example~\ref{ex:cancel} (2).
On the other hand,
 for $\alpha'_4=(-2,2,4)$,
 the pair consisting of $[1,0,0]$ on $\overline{\mathbb C P^2_{\alpha'_3}}$ 
 and $[1,0,0]$ on $\overline{\mathbb C P^2_{\alpha'_4}}$ 
 is a cancelling pair.
Hence,
 by taking weights $\alpha'_j$ for $1\leq j\leq16$ as
 \[
  \alpha'_j=\alpha'_k\qquad\textrm{for}\qquad
  j\equiv k\mod 4\qquad\textrm{and}\qquad 1\leq k\leq 4,
 \]
 16 $\mathbb Z_p$-manifolds 
 $\overline{\mathbb C P^2_{\alpha'_j}}$ $(1\leq j\leq16)$ have 
 12 cancelling pairs,
 that is,
 they satisfy the assumption of Proposition~\ref{prop:exist} 
 for $m=r=0$, $m'=16$ and $s=12$.
 
Since 
 \[
  N(11,(-1,1.2))=N(11,(-1,2,3))=N(11,(-1,3,4))=N(11,(-2,2,4))=1,
 \]
 if the above action is smooth 
 with respect to some smooth structure of $X$ 
 then
 \[
  \dim(\ind_{\mathbb Z_{11}}D)^{\mathbb Z_{11}}=6
 \]
 by Lemma~\ref{prop:ind}.
The action, therefore, does not satisfy 
 the inequality of Theorem~\ref{thm:smooth}
 because $b_2^+(X)$ is equal to 3.
This implies the action is nonsmoothable.
\end{proof}

\begin{corollary}
We have $11 \, \in \,NS(K3\, \#\, (\#^t S^2 \times S^2))$ for 
$t=1,2$ and $3$.
\end{corollary}
\begin{proof}
Let $X$ be $K3 \# (\#^t S^2 \times S^2)$ 
 for $t=1$, $2$, or $3$. 
Choose weights $\beta_k$ for $1\leq k\leq t$ arbitrarily, 
 then 
 the $\mathbb Z_p$-manifolds $\overline{\mathbb C P^2_{\alpha'_j}}$ for $1\leq j\leq 16$ 
 in the proof of Theorem~\ref{thm:K3} and 
 $S^4_{\beta_k}$ for $1\leq k\leq t$ satisfy 
 the assumption of Proposition~\ref{prop:exist} for $X$.

If the action is smoothable,
 \[
  \dim(\ind_{\mathbb Z_{11}}D_X)^{\mathbb Z_{11}}=6
 \]
 because the fixed points from $S^4_{\beta_k}$ does not
 contribute.
Then the inequality of Theorem~\ref{thm:smooth} is not satisfied, 
 which implies that the action is nonsmoothable.
\end{proof}

\begin{remark}
In the case of $n=2$ or $3$,
 the estimate of $p$ in Theorem~\ref{thm:sharp} 
 coincides with those in Theorem~\ref{thm:main}
 and it is not an improvement. 
In the case of $n\geq4$,
 Theorem~\ref{thm:sharp} gives an improvement.
Still better estimations might be obtained using
 the construction in Section~2
 using  other choices of $m$, $m'$, $r$ and $s$.
%
\end{remark}

\bibliographystyle{amsplain}

\end{document}